\documentclass[11pt]{amsart}
\usepackage{amscd,amsfonts,amsmath,amssymb,amstext,amsthm, mathrsfs}
\usepackage{setspace} 

\usepackage{graphicx, epic,eepic,bbm}
\usepackage{ragged2e}
\usepackage[normalem]{ulem}
\usepackage{color} 

\date{}
\def\dim{\operatorname{dim}}

\def\Ker{\operatorname{Ker}}

\def\Hom{\operatorname{Hom}}
\def\End{\operatorname{End}}

\def\ad{\operatorname{ad}} 
\def\M{\mathcal M}
\def\MS{\mathcal{M}S_C(n, L)}
\def\MO{\mathcal{M}O_C(n, L)}
\newcommand{\Sym}{\mathrm{Sym}}

\newcommand{\Iden}{\mathrm{Id}}
 
\newcommand{\rank}{\mathrm{rk}\,}
\newcommand{\Pic}{\mathrm{Pic}}

\newcommand{\Oc}{{{\cO}_{C}}}

\newcommand{\PP}{{\mathbb P}}
\newcommand{\Kc}{K_{C}}
\newcommand{\cE}{{\mathcal E}}
\newcommand{\cO}{{\mathcal O}}

\newcommand{\Ct}{C^{(2)}}

\newcommand{\Endo}{\mathrm{End}_0}

\newcommand{\IG}{\mathrm{IG}}
\newcommand{\rSU}{{\mathcal SU}}

\input xy
\xyoption{all}

\newtheorem{theorem}{{\textbf Theorem}}[section]
\newtheorem{proposition}[theorem]{{\textbf Proposition}}
\newtheorem{corollary}[theorem]{{\textbf Corollary}}
\newtheorem{lemma}[theorem]{{\textbf Lemma}}

\newtheorem{defn}[theorem]{{\textbf Definition}}
\newtheorem{notation}[theorem]{{\textbf Notation}}

\newtheorem{remit}[theorem]{{\textbf Remark}}

\newenvironment{remark}{\begin{remit}\rm}{\end{remit}}

  \title[Simplicity of Tangent bundles]{Simplicity of Tangent bundles on the moduli spaces \\of symplectic and orthogonal  bundles over a curve}

   \author[I. Choe]{Insong Choe$^\ast$}
 \address{Department of Mathematics, Konkuk University, 1 Hwayang-dong, Gwangjin-Gu, Seoul
143-701, Republic of Korea}
   \email{ischoe@konkuk.ac.kr}

   \author[G. H. Hitching]{George H.  Hitching}
 \address{Oslo Metropolitan University, Postboks 4, St. Olavs plass, 0130 Oslo, Norway}
   \email{gehahi@oslomet.no}

 \author[J. Hong]{Jaehyun Hong}
   \address{Center for Complex Geometry,  Institute for Basic Science (IBS), Daejeon 34126, Republic of Korea}
    \email{jhhong00@ibs.re.kr}

\subjclass[2020]{14D20, 53C10} 
\keywords{symplectic bundle, orthogonal bundle, minimal rational tangents}
\begin{document}

\begin{abstract}
The variety of minimal rational tangents associated to Hecke curves was used by J.-M.\ Hwang \cite{H01} to prove the simplicity of the tangent bundle on the moduli of vector bundles over a curve.
In this paper, we use the tangent maps  of the symplectic and orthogonal Hecke curves to prove an analogous result for symplectic and orthogonal bundles. In particular, we show the nondegeneracy of the associated variety of minimal rational tangents, which implies the simplicity of the tangent bundle on the moduli spaces of symplectic and orthogonal bundles over a curve. We also show that for large enough genus, the tangent map is an embedding for a general symplectic or orthogonal bundle.\\
\\
R\'{e}sum\'{e}. La vari\'{e}t\'{e} des tangentes des courbes minimales rationnelles associ\'{e}s aux courbes de Hecke, a \'{e}t\'{e} utilis\'{e}e par J.-M. Hwang \cite{H01} pour prouver la simplicit\'{e} du fibr\'{e} tangent \'{a} l’espace de modules des fibr\'{e}s vectoriels sur une courbe. Nous utilisons les applications tangentes des courbes de Hecke symplectiques et orthogonales pour demontrer un resultat analogue pour les fibr\'{e}s symplectiques et orthogonaux. En particulier, nous prouvons que la vari\'{e}t\'{e} des tangentes des courbes rationnelles minimales associ\'{e}e est nondeg\'{e}n\'{e}r\'{e}e ; ce qui implique la simplicit\'{e} des fibr\'{e}s tangents des espaces de modules des fibr\'{e}s symplectiques et orthogonaux sur une courbe. Nous montrons d’ailleurs, pour genre suffisamment grand, que l'application tangente est un plongement pour un fibr\'{e} symplectique ou orthogonal g\'{e}n\'{e}rique.
\end{abstract}

 \maketitle


\section{Introduction}

Let $C$ be a smooth   projective curve of genus $g \ge 2$ over the complex numbers. Let $\M:=\rSU_C(n,d)$ be the moduli space of semistable vector bundles over $C$ of rank $n$ with fixed determinant of degree $d$. 
Note that $\M$ is a Fano variety of Picard number 1 and moreover smooth if  $n$ and $d$ are coprime.  

It was shown in \cite[Corollary 1]{H01} that  for $g \ge 4$, the tangent bundle of  the smooth part $\M^\circ \subset \M$ is simple. 
The strategy   was to  exploit certain minimal rational curves called Hecke curves and the associated variety of minimal rational tangents. 
More precisely,  it is shown that the variety of minimal rational tangents $\mathcal{C}_W$ at a generic point $W \in \M$ is non-degenerate in $\mathbb{P} (T_W \M)$ and this implies the simplicity of the tangent bundle (cf.\ Proposition \ref{nondegsimple}).

The goal of this paper is to prove the analogous result for the moduli spaces $\MS$ of symplectic bundles and $\MO$ of orthogonal bundles. The symplectic and orthogonal versions of  Hecke curves were constructed   in \cite{CCL22}. In the same paper, these curves were shown to be the minimal rational curves in the ambient varieties. Based on this, we establish the following results in this paper:
\begin{itemize}
\item The smoothness of the symplectic and orthogonal Hecke curves (\S 3)
\item The nondegeneracy of the tangent map on the variety of minimal rational tangents of these Hecke curves (\S 4)
\item The very-ampleness of the associated complete linear system  (\S 5)
\end{itemize}
In particular, as a corollary of the nondegeneracy  in \S 4, we show that the tangent bundles of $\MS$ and $\MO$ are simple, under a certain genus bound.

 We would like to add a word of warning for the arguments that will follow. Inside the moduli space $\rSU_C(n,d)$ of vector bundles, the locus of symplectic/orthogonal bundles form a closed subvariety, and  a  symplectic/orthogonal Hecke curve can be thought of as either a special kind of  Hecke curve on $\rSU_C(n, d)$ or its variation. So one might expect that 
 the results for Hecke curves in \cite{H01}, \cite{HR04}, or \cite{S05}  directly imply the same results for the symplectic or  orthogonal setting. 
 
 But in most discussions of a Hecke curve on $\rSU_C(n, d)$, 
  one assumes that it passes through a generic point, such as a $(1,1)$-stable bundle \cite[Definition 5.1]{NR78}. And it is unclear if a generic point of $\MS$ and/or $\MO$  is $(1,1)$-stable as a vector bundle. By simple dimension comparison, it is still possible that the subvarieties $\MS$ and/or $\MO$ are entirely contained in the  non-$(1,1)$-stable  locus. For this reason, we cannot tell from the outset if the symplectic and orthogonal Hecke curves share the same properties as the Hecke curves passing through $(1,1)$-stable locus. This is why we later devise arguments  based on $\delta$-stability (cf.\ \cite[\S 4.1]{CCL22}) on $\MS$ and $\MO$.



\section{Preliminary results}

In this section, we gather the notations and preliminary results relevant to our discussion.  Let $C$ be a smooth  projective curve of genus $g \geq 2$. 
\begin{notation}
Given a subspace $\Lambda \subset V^*$ of a dual vector space, $\Lambda^\perp \subset V $ denotes the kernel:
\[
\Lambda^\perp = \{ v \in V \: : \: \lambda(v) = 0 \ \text{for all } \lambda \in \Lambda \}.
\]
Also for a subspace $U \subset V$, $U^\perp \subset V^*$ denotes the annihilator: 
\[
U^\perp = \{ \phi \in V^* \: : \: \phi ( u )  =  0 \ \text{for all } u \in U \}.
\]
When $V$ is equipped with a bilinear form $\omega\colon V \otimes V \to \mathbb C$,  we define
\[
\ker (\omega) := \{ v_0 \in V \: : \: \omega (v_0, v) = 0 \ \text{for all } v \in V \}
\]
\end{notation}

\subsection{Hecke modification} \label{Heckemodification}  Let $W$ be a vector bundle over $C$.
Choose a subspace $\Lambda \subset W_x^*$ for some $x \in C$.
The  {\it Hecke modification  $W^\Lambda$  of $W$ along $\Lambda$}  is given by the kernel of the composition map $W \rightarrow W_x \rightarrow W_x/\Lambda^\perp$.
 There is an exact sequence of sheaves:
$$ 0 \rightarrow  W^\Lambda \stackrel{\phi}{\rightarrow} W \rightarrow (W_x/\Lambda^\perp) \otimes \mathcal O_x \rightarrow 0 $$
whose restriction to the fiber at $x$ is given by

$$0 \rightarrow \Ker(\phi_x\colon W_x^\Lambda \rightarrow W_x)\rightarrow W^\Lambda_x \rightarrow W_x \rightarrow W_x/\Lambda^\perp \rightarrow 0.$$
Then the locally free sheaf $W^\Lambda$ corresponds to a vector bundle with $\det (W^\Lambda) = \det(W) \otimes \mathcal O_C(-kx)$, where $k $ is the dimension of $\Lambda$ in $W_x^*$.

\subsection{Hecke curves on $\rSU_C(n,d)$} \label{heckecurves} The main reference for this subsection is \cite{H01}.

Let  $W$ be a vector bundle over $C$.
For a subspace $\theta \subset W_x^*$ of dimension one,  by abuse of notation,
let $W^{\theta}$ denote the Hecke modification of $W$ along $\theta$. Then we have
$$0 \rightarrow W^{\theta} \rightarrow W \rightarrow (W_x/\theta^{\perp}) \otimes \mathcal O_x \rightarrow 0.$$

For a subspace $\ell \subset W^{\theta}_x$ of dimension one,
 the Hecke modification $V^{\ell} $  of $V:=(W^{\theta})^*$ along $\ell$ fits into the exact  sequence
\begin{equation} \label{Heckeseq}
0 \rightarrow V^{\ell} \stackrel{\beta}{\rightarrow}  V \rightarrow (V_x/\ell^{\perp}) \otimes \mathcal O_x \rightarrow 0. 
\end{equation}
In particular for $\ell_0:=\Ker(W_x^{\theta} \rightarrow W_x)$,
we have $(V^{\ell_0})^* \simeq W$.

For any two-dimensional subspace $U$ with   $\ell_0 \subset U \subset W^{\theta}_x$,  the subspace $U^{\perp}$  has codimension two in  $V_x$, and is contained in $\ell_0^{\perp}$. Hence the family 
\[
\{(V^{\ell})^* \: : \: { [\ell] \in \mathbb P(U)   } \}
\]
parameterized by 
 $\mathbb P(U) =\mathbb P(V_x/U^{\perp}) \cong \mathbb P^1$, is a deformation of $W$. 
 If we choose a generic  $W \in \rSU_C(n,d)$, then this family gives a smooth  rational curve through $W$ called a \textit{Hecke curve}. 
 It was shown in  \cite{S05} that the Hecke curves have minimal degree among the rational curves passing through a generic $W \in \rSU_C(n,d)$. 
 
Note that the parameter space of Hecke curves passing through $W \in \rSU_C(n,d)$ is given by a double fibration 
 \[
 \PP (T_\pi^*) \to \PP (W^*) \stackrel{\pi}{\to} C,
 \]
 where $T_\pi$ is the vertical tangent bundle of $\pi \colon \PP (W^*) \to C$. In the previous notation, this corresponds to the composition map 
 \[
 (U  / \ell_0 \subset W_x^\theta/ \ell_0)  \mapsto (\theta  \subset W_x^*) \mapsto x.
 \]
 In particular for $n=2$, this double fibration boils down to the ruled surface $\PP (W^*)$.

 \subsection{Kodaira--Spencer map} \label{KSell}
The main references of this subsection are \cite{NR75} and \cite{NR78}. Consider the above family $\{ V^\ell  :  [\ell] \in  \mathbb P (U) \}$ as a deformation of $V^{\ell_0} \cong W^*$. Since the map  $\beta_x \colon V_x^{\ell_0} \to V_x$ surjects onto $\ell_0^\perp$, we have the induced pull-back map 
\[
\widehat{\beta_x}\colon  \Hom (\ell_0^\perp / U^\perp, V') \to \Hom (V_x^{\ell_0}, V')
\]   
for any vector space $V'$.
\begin{proposition} \label{KScomp} 
The Kodaira--Spencer map of the family $\{ V^\ell  :  [\ell] \in  \mathbb P (U) \}$  is given by
\begin{eqnarray*}
\Hom(\ell_0^{\perp} /U^\perp , V_x/\ell_0^{\perp}) \stackrel{\widehat{\beta_x}}{\longrightarrow} \Hom(V_x^{\ell_0}, V_x/\ell_0^{\perp}) \stackrel{\delta}{\rightarrow}  H^1(C, \mathrm{End}(V^{\ell_0})),
\end{eqnarray*}
where $\delta$ is induced from the sequence (\ref{Heckeseq}).
\end{proposition}

Later we will also need to consider a slight generalization of the above family, where the Hecke modification is taken for subspaces of dimension two. 
For a subspace $\Theta \subset W_x^*$ of dimension 2, the Hecke modification $W^{\Theta}$ of $W$ along $\Theta$ can be put into the following exact sequence:
$$ 0\rightarrow W^{\Theta} \rightarrow W \rightarrow (W_x/\Theta^{\perp}) \otimes \mathcal O_x \rightarrow 0.$$
 Let  $\wp_0$ denote the kernel of $W_x^{\Theta} \rightarrow W_x$. The following is a straightforward generalization of Proposition \ref{KScomp}.
\begin{proposition}   \label{prop: KS map codimension 4}
  Let $U$ be a subspace  of dimension   4 with $\wp_0 \subset U \subset  W_x^{\Theta} $; that is, $U^{\perp} \subset \wp_0^{\perp} \subset (W_x^{\Theta})^*=:V_x$. Then
the family $\{V^{\wp} : [\wp] \in   Gr(2,  V_x/U^{\perp}) \}$ is a deformation of $W^* =V^{\wp_0}$, and its  Kodaira--Spencer map
$$T_{\wp_0}(Gr(2,V_x/U^{\perp})) = \Hom(\wp_0^{\perp}/U^{\perp}, V_x/\wp_0^{\perp}) \rightarrow H^1(C, \End(V^{\wp_0}))$$ is given by the composition
\begin{eqnarray*}
\Hom(\wp_0^{\perp}/U^{\perp}, V_x/\wp_0^{\perp}) \stackrel{\widehat{\beta_x}}{\rightarrow} \Hom(V_x^{\wp_0}, V_x/\wp_0^{\perp}) \stackrel{\delta}{\rightarrow}  H^1(C, \mathrm{End}(V^{\wp_0})),
\end{eqnarray*}
where  $\widehat{\beta_x} $ and $\delta$ are induced from
$$0 \rightarrow V^{\wp_0} \stackrel{\beta}{\rightarrow} V \rightarrow (V_x/\wp_0^{\perp}) \otimes \mathcal O_x \rightarrow 0.$$
\end{proposition}

 \subsection{Symplectic Hecke curves on $\MS$}
In this subsection, we recall the construction of  symplectic Hecke curves, following \cite{CCL22},  to which we refer the reader for the details.  
 
For a line bundle $L$ on $C$,  an $L$-valued symplectic bundle of rank $n$ is  a vector bundle  $W$  of (even) rank $n$ equipped with an $L$-valued symplectic form
$\omega\colon W \otimes W \rightarrow L.$  Let $\MS$  be  the moduli space of $L$-valued    symplectic bundles of rank $n$. By the morphism forgetting the symplectic forms,  the moduli space $\MS$ can be thought of as a subvariety of $\rSU_C(n, \frac{1}{2} n \ell)$, where $\ell = \deg (L)$. To avoid the coincidence $\mathcal M S_C(2,L) = \rSU_C(2, \ell)$, we assume $n \ge 4$ throughout this paper. 

The construction of symplectic Hecke curves on $\MS$ closely follows  the previous construction of Hecke curves on $\rSU_C(n,d)$, keeping track of the  deformation of  symplectic forms.
For a subspace $\theta \subset W_x^*$ of dimension one, let $W^\theta$ be the Hecke modification of $W$ along $\theta $, fitting into the sequence
\begin{equation*} \label{exacttheta}
0 \rightarrow W^{\theta} \rightarrow W \rightarrow (W_x/\theta^{\perp}) \otimes \mathcal O_x \rightarrow 0.
\end{equation*}
Noting that every 1-dimensional subspace $\theta$ is isotropic, we get an induced $L^*(x)$-valued skew-symmetric form on $V:=(W^{\theta})^*$:
$$\omega^{\theta}\colon(W^{\theta})^* \rightarrow W^{\theta} \otimes L^*(x) .$$
 Then $\ker \omega^{\theta}_x$ has codimension two in $V_x=(W^{\theta})^*_x$.

For a subspace $\ell \subset  W^{\theta} _x$ of dimension one,   we have the sequence
$$0 \rightarrow V^{\ell} \rightarrow V  \rightarrow (V_x/\ell^{\perp}) \otimes \mathcal O_x \rightarrow 0.$$
 Then the bundle 
   $V^{\ell}$ has a skew-symmetric form  induced from $\omega^{\theta}$,  and it is an $L^*$-valued  nondegenerate (symplectic) form if and only if $\ker \omega^{\theta}_x \subset \ell^{\perp}$.
Now the family 
\[
\{V^\ell \: : \:  \ker \omega^{\theta}_x \subset \ell^{\perp} \subset V_x\}
\]
 of $L$-valued symplectic bundles are parameterized by 
$\mathbb P(V_x/\ker \omega^{\theta}_x) \simeq \mathbb P^1$.    In particular, if   $\ell_0$ is in the kernel of $W_x^{\theta} \rightarrow W_x$, then $V^{\ell_0} \cong W^*$.

 Under the assumption that $g \ge 3$ and $W \in \MS$ is a generic point, by \cite[Lemma 4.5]{CCL22} the dual family 
$$
\{(V^{\ell})^* \: : \: \ell \in \mathbb P(V_x/\ker \omega^{\theta}_x) \}
 $$
gives a    rational curve on $\MS$ passing through $W$, called a \textit{symplectic Hecke curve}.
Also it was shown  in \cite[Theorem 5.2]{CCL22} that these  curves have minimal degree among the rational curves passing through a generic point $W \in \MS$.

 Later we need the following fact.

\begin{proposition} \label{prop: contained in smooth locus symplectic} Assume $g \ge 3$. 
For a generic point $W \in \MS$, every symplectic Hecke curve  passing through $W$ is contained in the smooth locus of $\MS$. 
\end{proposition}
\begin{proof}
By \cite[Lemma 4.5]{CCL22}, any symplectic Hecke curve passing through a generic  point $W$ stays inside the locus of  stable symplectic bundles. To see that it is contained in the smooth locus, it suffices to show that it does not touch the locus of non-regularly stable symplectic bundles which are of the form $W_1 \perp W_2$ for some stable symplectic subbundles $W_1$ and $W_2$. 

This can be checked by dimension count: The locus of non-regularly stable bundles is contained in a finite union of the images of $\mathcal M S_C(n_1, L) \times \mathcal M S_C(n_2, L)$, where $n_1+n_2 = n$. Since  symplectic Hecke curves passing through a point $W_0$ in this subvariety are parameterized by $\mathbb P (W_0^*)$, it suffices to check the inequality:
\[
\dim \mathbb{P}  (W_0^*)  + \dim \mathcal M S_C(n_1, L) + \dim  \mathcal M S_C(n_2, L) \ < \ \dim \MS
\]
for any even integers  $n_1, n_2$ with $n_1+n_2 = n$. This boils down to $n < n_1n_2 (g-1) $, which holds for $g \ge 3$. 
\end{proof}

 \subsection{Orthogonal Hecke curves on $\MO$} 
Again, the main reference in this subsection for a construction of  orthogonal Hecke curves will be  \cite{CCL22}.
 
 For a line bundle $L$ on $C$,  an $L$-valued orthogonal bundle of  rank $n$  is a vector bundle $W$ of rank $n$ equipped with an $L$-valued orthogonal form
$b\colon W \otimes W \rightarrow L.$  Let $\MO$  be  the moduli space of $L$-valued    orthogonal bundles of rank $n$. The moduli space $\mathcal M O_C(2r, L)$ has several irreducible components, due to the invariants $ \det (W)$ and the 2nd Stiefel--Whitney class $w_2(W)$ (see  \cite[\S2]{CCH21}).  
 By the morphism forgetting the orthogonal forms,  each irreducible component of the moduli space $\MO$ is sent to  a subvariety of $\rSU_C(n, \frac{1}{2} n \ell)$, where $\ell = \deg (L)$.

As in \cite{CCL22}, we assume $n \ge 5$ throughout this paper. The reason behind this convention is that the moduli space $\mathcal M O_C(n, L)$ has Picard number one for $n \ge 5$, while  $\mathcal M O_C(4, L)$ has Picard number two. Accordingly, the minimality of the orthogonal Hecke curves was discussed in \cite{CCL22} for $n \ge 5$.
We remark that there is a standard construction of orthogonal bundles of low rank from vector bundles, described in \cite{CH22}.

The construction of orthogonal Hecke curves on $\MS$ is a little bit different from that of Hecke curves on $\rSU_C(n,d)$:  the dimension of the involved subspaces are doubled.

For an isotropic subspace $\Theta \subset W_x^*$ of dimension two, let $W^\Theta$ be the Hecke modification:
$$0 \rightarrow W^{\Theta} \rightarrow W \rightarrow (W_x/\Theta^{\perp}) \otimes \mathcal O_x \rightarrow 0.$$
  Then there is an $L^*(x)$-valued  symmetric form
$$b^{\Theta}\colon(W^{\Theta})^* \rightarrow W^{\Theta} \otimes L^*(x) $$
on $V:=(W^{\Theta})^*$ such that   $\ker b^{\Theta}_x$ has codimension four in $V_x=(W^{\Theta})^*_x$.

For an isotropic  subspace $\wp \subset  W^{\Theta} _x$ of dimension  two with  $\wp^{\perp} \subset V_x$, we have the Hecke modification
$$0 \rightarrow V^{\wp} \rightarrow V  \rightarrow (V_x/\wp^{\perp}) \otimes \mathcal O_x \rightarrow 0.$$
 Then $V^{\wp}$ is equipped with a symmetric form  induced from $b^{\Theta}$, and it is an  
 $L^*$-valued  nondegenerate (orthogonal) form if and only if $\ker b^{\Theta}_x \subset \wp^{\perp}$. 
 %
  In particular   when $\wp_0$ is in the kernel of $W_x^{\Theta} \rightarrow W_x$, we have $V^{\wp_0} \cong W^*$.
  
Note that the space of two-dimensional isotropic subspaces $\wp  \subset W_x^\Theta$  such that 
\[
 \ker b^{\Theta}_x \ \subset \ \wp^{\perp} \ \subset  \ V_x \ = \ (W^{\theta})^*_x
\]
is the isotropic  Grassmannian of 2-dimensional subspaces of  $V_x/\ker b_x^{\Theta} \cong \mathbb C^4$, which is  a disjoint union  of two  projective lines. Let $\mathrm{IG}(2, V_x/\ker b_x^{\Theta})$ be the line  containing the point $\wp_0^\perp / \ker b_x^{\Theta}$. 
 Then the family 
$$ \{ V^{\wp} \: : \: \wp  \in \IG(2, V_x/\ker b_x^{\Theta})  \}$$
of $L^*$-valued orthogonal bundles gives a deformation of $W^*$. 

Under the assumption that $g \ge 5, n \ge 5$ and  that $W \in \MO$ is a generic point, by \cite[Lemma 4.7]{CCL22} the dual family $\{  ( V^{\wp})^* \}$ gives a    rational curve on $\MO$ passing through $W$, called an \textit{orthogonal Hecke curve}.
Also it was shown  in \cite[theorem 5.3]{CCL22} that these  curves have minimal degree among the rational curves passing through a generic point $W \in \MO$.

Again, we show the following.
\begin{proposition} \label{prop: contained in smooth locus orthogonal} Assume $g \ge 5$ and $n \ge 5$. 
For a generic point $W \in \MO$, every orthogonal Hecke curve  passing through $W$ is contained in the smooth locus of $\MO$. 
\end{proposition}
\begin{proof}
By \cite[Lemma 4.7]{CCL22}, any orthogonal Hecke curve passing through a generic point $W$ stays inside the locus of  stable orthogonal bundles. To see that it is contained in the smooth locus, it suffices to show that it does not touch the locus of non-regularly stable orthogonal bundles which are of the form $W_1 \perp W_2$ for some stable orthogonal subbundles $W_1$ and $W_2$. 

This can be checked by dimension count: The locus of non-regularly stable bundles is contained in a finite union of the images of $\mathcal M O_C(n_1, L) \times \mathcal M O_C(n_2, L)$, where $n_1+n_2 = n$. Since orthogonal Hecke curves passing through a point $[W_0]$ in this subvariety are parameterized by $\IG(2, W_0^*) $, it suffices to check the inequality:
\[
\dim \IG(2,  W_0^*)  + \dim \mathcal M O_C(n_1, L) + \dim  \mathcal M O_C(n_2, L) \ < \ \dim \MO
\]
for any  integers  $n_1, n_2$ with $n_1+n_2 = n$. Since $\dim IG(2, W_0^*) = 2n-6$,  this boils down to $2n-6 < n_1n_2 (g-1) $, which holds for $g , n \ge 5$. 
\end{proof}

\subsection{Minimal rational curves}

 Let $M$ be a projective variety. Let $ \mathcal K$ be an irreducible component of the Hilbert scheme of complete curves on $M$ such that  generic members of $\mathcal K$   cover an open subset of the smooth locus  $M^\circ$ of $M$.  For a generic point $x \in M$, denote by $\mathcal K_x$ the subscheme  of $\mathcal K$ consisting of members of $\mathcal K$ passing through $x$.  Assume that  for a generic point $x \in M$, every member of $\mathcal K_x$ is an irreducible smooth rational curve contained in  $M^\circ$ and $\mathcal K_x$ is an irreducible  complete variety.
 In this case, we call $\mathcal K$ a {\it minimal rational component} of $M$.
 
A covering family of rational curves having minimal degree gives a minimal rational component. More precisely, an irreducible component  $\mathcal K$ is a minimal rational component of $M$ if it satisfies the following conditions:
\begin{enumerate}
\item[(i)] For a generic $x \in M$, every member of $\mathcal K_x$  is irreducible smooth rational curve  contained in $M^\circ$.
\item[(ii)] The locus swept out by the curves in $\mathcal K$ is dense in $M$.
\item[(iii)] For a fixed ample line bundle $\xi$ on $M$, the degree of members of $\mathcal K$ with respect to $\xi$ is minimal among the curves in an irreducible family satisfying (1) and (2). 
\end{enumerate}

Let $x \in M$ be a generic point. Define the tangent map $\tau_x\colon \mathcal K_x \dashrightarrow \mathbb P(T_x(M))$ by 
 $$\tau_x([R]) =[T_xR] \in \mathbb P(T_x(M)),$$ 
 where $R$ is a smooth rational curve in $M^\circ$ passing through $x$. The closure $\mathcal C_x$ of the image ${\tau_x (\mathcal K_x )} \subset  \mathbb P(T_x(M))$ is called the {\it variety of minimal rational tangents} (VMRT for short) at $x$ associated with $\mathcal K$.

 The following  is \cite[Theorem 2]{H01}, which connects the theory of VMRT and the simplicity of the tangent bundle.
  \begin{proposition} \label{nondegsimple}
 Let $M$ be a Fano variety which has a minimal rational component $\mathcal K$. 
 If the VMRT  $\mathcal C_x$ at a generic point $x \in M$ is non-degenerate in $\mathbb P T_xM$, then the tangent bundle $T(M^\circ)$ is simple.
 \end{proposition}

For  the moduli space of vector bundles $\M = \rSU_C(n, d)$, it is proven in \cite{H01} that for $g \ge 4$, the irreducible  component  $\mathcal K$  of the Hilbert scheme of $\mathcal M$ containing  Hecke curves is a minimal rational component of $\M$. In this case,  given  a generic point $W \in  \mathcal M$,  $\mathcal K_{W}$ is given by
\[
\mathcal K_{W} = \bigcup_{[\theta] \in \mathbb P(  W ^*)} \mathbb P(W_x^{\theta}/\ell_0) 
\ \simeq \ \mathbb P(T_{\pi}^*) 
\]
where $T_{\pi}$ is the vertical tangent bundle of $\pi\colon \mathbb P(W^*) \rightarrow C$.
In particular for $n= 2$, we have $\mathcal K_{W} = \mathbb P(W^*)$. 

Moreover it is shown in \cite[Theorem 3.1, Theorem 3.7]{HR04} that the tangent map at a generic point $W \in \M$ is biregular to the image for $g \ge 5$ and birational  for $g =4$.

Finally we discuss the case of $\MS$ and $\MO$.
By the result \cite[Theorem 5.2]{CCL22} on the minimality of degree, we can see that there is a minimal rational component $\mathcal K$ of $\MS$ containing symplectic Hecke curves such that $\mathcal K_{W}$  for a   generic element $W$  is given by
\begin{eqnarray*}
\mathcal K_{W}& =&  \mathbb P(W^*). 
\end{eqnarray*}
 
 Similarly by  \cite[Theorem 5.3]{CCL22},  there is a minimal rational component $\mathcal K$ of $\MO$ containing orthogonal Hecke curves such that $\mathcal K_{[W]}$  for a  generic element $[W] $  is given by
\[
\mathcal K_{W} =  IG(2, W^*). 
\]

 \section{Smoothness of    Hecke curves}
 In this section, we show the smoothness of the symplectic and orthogonal Hecke curves. 
\begin{proposition} \label{symplecticsmooth}
Assume $g \ge 4$ and $n \ge 4$.
  Then any symplectic Hecke curve passing through a generic point $W \in \MS$ is smooth.
\end{proposition}

\begin{proof}
 From the construction, the family 
 \begin{equation} \label{sympHecke}
\{V^\ell \: : \:  \ker \omega^{\theta}_x \subset \ell^{\perp} \subset V_x\}
\end{equation}
  gives  a deformation of $V^{\ell}$ along a subspace $U:=(\ker \omega_x^\theta)^\perp \cong \mathbb C^2$. 
   By Proposition \ref{KSell}, the Kodaira--Spencer map 
\begin{equation} \label{KSsymplectic}
T_{\ell}(\mathbb P(V_x/U^{\perp})) = \Hom(\ell^{\perp}/U^{\perp}, V_x/\ell^{\perp}) \rightarrow H^1(C, \End(V^{\ell}))
\end{equation}
associated to 
this family is given by  the composition
\begin{eqnarray}  
\Hom(\ell^{\perp}/U^{\perp}, V_x/\ell^{\perp}) \stackrel{\widehat{\beta_x}}{\rightarrow} \Hom(V_x^{\ell}, V_x/\ell^{\perp}) \stackrel{\delta}{\rightarrow}  H^1(C, \mathrm{End}(V^{\ell})),
\end{eqnarray}
where  $\widehat{\beta_x} $ and $\delta$ are induced from
\begin{equation} \label{betaseq}
0 \rightarrow V^{\ell} \stackrel{\beta}{\rightarrow} V \rightarrow (V_x/\ell^{\perp}) \otimes \mathcal O_x \rightarrow 0.
\end{equation}
Note that this  can be geometrically understood as the composition map 
 \[
 T_{\ell} (\mathbb P (V_x / U^\perp)) \stackrel{\widehat{\beta_x}}{\longrightarrow } T_{[\beta]} \mathrm{Quot} (V)\stackrel{\delta}{\longrightarrow} H^1(C, \ad (V^{\ell})) = T_{[V^{\ell}]} \mathcal M S_C(n, L^*),
 \]
 where    the tangent space $T_{[\beta]} \mathrm{Quot} (V)$ of the Quot scheme of $V$ is given by 
 \[
 T_{[\beta]} \mathrm{Quot} (V) = H^0(C, \mathrm{Hom} (V^{\ell}, (V_x/\ell^{\perp}) \otimes \mathcal O_x) ) \cong \mathrm{Hom} (V_x^{\ell}, V_x/\ell^{\perp}).
 \]
 Hence to show that the map 
 \[
 \phi_U\colon \mathbb P (U) \cong \mathbb P^1 \to \MS
 \]
 which gives the symplectic Hecke curve (\ref{sympHecke}) is an immersion, we need to show that the map (\ref{KSsymplectic}) is  injective.  
 
 Since $\widehat{\beta_x}$ is injective by definition of Quot schemes, we need to check that $\delta$ is injective. The map $\delta$ fits into the long exact sequence associated to (\ref{betaseq}) tensored by $(V^{\ell})^*$:
 \[
 0 \to  H^0(C, \mathrm{End}(V^{\ell})) \to H^0( C, (V^{\ell})^* \otimes V ) \to \mathrm{Hom} (V^{\ell}, (V_x/\ell^{\perp})  ) \stackrel{\delta}{\to}   H^1(C, \mathrm{End}(V^{\ell})) 
 \]
   Hence $\delta$ is injective everywhere if we know:
 \begin{itemize}
 \item $\dim H^0( C, \mathrm{End}(V^{\ell}) ) =1$ for all $\ell$ and
 \item $\dim H^0( C, (V^{\ell})^* \otimes V )  =1$ for all $\ell$.
 \end{itemize}
 The first condition holds if $V^{\ell}$ is  regularly stable. By \cite[Lemma 4.2]{CCL22}, the second condition holds 
 if every point $V$ is a generic point   and $g \ge 3$.
 
 Now it remains to show that the map $\phi_U$ is injective. It was shown  in \cite[Lemma 4.5]{CCL22} that $\phi_U$ is generically injective if $g \ge 3$ and $n \ge 4$. Its proof can be slightly modified to show the injectiveness (under a stronger bound on $g$ and $n$).
 The point of the proof was to choose $W = (V^{\ell_0})^*$ as a ``1-stable'' symplectic bundle (see \cite[\S 4.1]{CCL22}). By the same argument, if we choose $W$ to be  2-stable, then every bundle $V^\ell$ can be shown to be 1-stable, and hence $V^{\ell_1} \cong V^{\ell_2}$ implies $\ell_1 = \ell_2$. By \cite[Lemma 4.1]{CCL22} a generic point of $\MS$ is 2-stable for $g \ge 4, n \ge 4$ and we are done.
\end{proof} 
 
\begin{proposition} \label{orthosmooth}
Assume $g \ge 5$ and $n \ge 5$.
  Then any orthogonal Hecke curve passing through a generic point $W \in \MO$ is smooth.
\end{proposition}

\begin{proof} 
From the construction of orthogonal Hecke curves, the family 
 \begin{equation} \label{orthHecke}
\{V^\wp \: : \:  \Ker \omega^{\Theta}_x \subset \wp^{\perp} \subset V_x\}
\end{equation}
  gives  a deformation of $V^{\wp}$ along a subspace $U:=(\Ker \omega_x^\Theta)^\perp \cong \mathbb C^4$.  By Proposition \ref{prop: KS map codimension 4}, the Kodaira--Spencer map 
\begin{equation} \label{KSortho}
T_{\wp}Gr(2, V_x/U^{\perp}) = \Hom(\wp^{\perp}/U^{\perp}, V_x/\wp^{\perp}) \rightarrow H^1(C, \End(V^{\wp}))
\end{equation}
associated to 
this family is given by  the composition
\begin{eqnarray}  
\Hom(\wp^{\perp}/U^{\perp}, V_x/\wp^{\perp}) \stackrel{\widehat{\beta_x}}{\rightarrow} \Hom(V_x^{\wp}, V_x/\wp^{\perp}) \stackrel{\delta}{\rightarrow}  H^1(C, \mathrm{End}(V^{\wp})),
\end{eqnarray}
where  $\widehat{\beta_x} $ and $\delta$ are induced from
\begin{equation} \label{betaseq2}
0 \rightarrow V^{\wp} \stackrel{\beta}{\rightarrow} V \rightarrow (V_x/\wp^{\perp}) \otimes \mathcal O_x \rightarrow 0.
\end{equation}
Note that this  can be geometrically understood as the composition map 
 \[
 T_{\ell} Gr(2,V_x / U^\perp) \stackrel{\widehat{\beta_x}}{\longrightarrow } T_{[\beta]} \mathrm{Quot} (V)\stackrel{\delta}{\longrightarrow} H^1(C, \ad (V^{\wp})) = T_{[V^{\wp}]} \mathcal M O_C(n, L^*),
 \]
 where    the tangent space $T_{[\beta]} \mathrm{Quot} (V)$ of the Quot scheme of $V$ is given by 
 \[
 T_{[\beta]} \mathrm{Quot} (V) = H^0(C, \mathrm{Hom} (V^{\wp}, (V_x/\wp^{\perp}) \otimes \mathcal O_x) ) \cong \mathrm{Hom} (V_x^{\wp}, V_x/\wp^{\perp}).
 \]
 Hence to show the immersedness of the orthogonal Hecke curve (\ref{orthHecke}), it suffices to show that the map (\ref{KSortho}) is  injective. 
 
 Since $\widehat{\beta_x}$ is injective by definition of Quot schemes, we need to check that $\delta$ is injective. The map $\delta$ fits into the long exact sequence associated to (\ref{betaseq2}) tensored by $(V^{\wp})^*$:
 \[
 0 \to  H^0(C, \mathrm{End}(V^{\wp})) \to H^0( C, (V^{\wp})^* \otimes V ) \to \mathrm{Hom} (V^{\wp}, (V_x/\wp^{\perp})  ) \stackrel{\delta}{\to}   H^1(C, \mathrm{End}(V^{\wp})) 
 \]
   Hence $\delta$ is injective everywhere if we know:
 \begin{itemize}
 \item $\dim H^0( C, \mathrm{End}(V^{\wp}) ) =1$ for all $\wp$ and
 \item $\dim H^0( C, (V^{\wp})^* \otimes V )  =1$ for all $\wp$.
 \end{itemize}
 The first condition holds if $V^{\wp}$ is  regularly stable. The second condition holds 
 if $V$ is general and $g \ge 3$ by \cite[Lemma 4.2]{CCL22}.
 
 Now to show that the injectiveness.    as in the symplectic case,
it suffices to  choose $W$ to be 3-stable in order that every bundle $V^\wp$ is 2-stable. By \cite[Lemma 4.1]{CCL22} a generic point of $\MO$ is 3-stable for $g \ge 5, n \ge 5$ and we are done.
\end{proof}

\section{Nondegeneracy of the tangent map}

In this section, we discuss the tangent map   of the variety of minimal rational tangents for the moduli spaces $\M S_C(2r, L) $ and $\MO$. 
In particular, we study the complete linear system which defines the tangent map. This confirms that the image of the tangent map is nondegenerate, and as a consequence we get the simpleness of  the tangent bundle of the moduli space. Basically we follow the computations in \cite{H00}  of the Kodaira--Spencer map of the Hecke curves on the moduli space $\mathcal \rSU_C(2,d)$.

\subsection{Symplectic   bundles}

The tangent space of $\rSU_C(n,d)$ at a stable bundle $W$ is given by $H^1(C, \Endo (W))$, where $\Endo (W)$ is the vector bundle of traceless endomorphisms of $W$. By a similar argument as \cite[Lemma 2.2]{BH21}, the tangent space of $\MS$ at a regularly stable symplectic bundle $W$ is given by $H^1(C, L \otimes \mathrm{Sym}^2W^*)$. In this context, we put $\ad (W) = L \otimes \mathrm{Sym}^2W^*$.

\begin{proposition} \label{prop: tangent map symplectic}
Assume $g \ge 4$ and $n \ge 4$ as in Proposition \ref{symplecticsmooth}.
  Let $\mathcal K$ be  the minimal rational component consisting of symplectic  Hecke curves on  
  $\M S_C(2r, L)$.
%
Then for a generic $W \in \M S_C(2r,L)$, the tangent map 
$$\tau_{W} \colon \mathcal K_{W}=\mathbb P W^*  \to \mathbb P T_{W} \MS = \mathbb P H^1(C, \mathrm{Sym}^2 W^* \otimes L) $$
 is the composition $\Phi_W \circ \iota$,
where $\Phi_W$ is given by the complete linear system $|\mathcal O(1) \otimes \pi^* K_C|$.
Also, $\iota(\theta) = \theta \otimes (\omega(\theta, \,\cdot\,))$  for $\theta \in W_x^*$, 
 and  the image of $\iota(\theta) $ in $\mathbb P(H^1(C, \mathrm{Sym}^2W^* \otimes L))$ is given by the linear functional $H^0(C, \mathrm{Sym}^2 W \otimes L^* \otimes K_C) \rightarrow K_{C,x}$ taking the trace of endomorphisms of $W_x$.
\begin{eqnarray*}
\xymatrix@R+1pc@C+2pc{
\mathcal K_{W}=\mathbb P W^*  \ar[dr]_{\iota} \ar[r] & \mathbb P(\Endo(W))  \ar@{-->}[r]_{\Psi_W}  & \mathbb P H^1(C, \Endo W) \\
 & \mathbb P(\mathrm{Sym}^2 W^* \otimes L ) \ar@{^{(}->}[u] \ar[r]<4pt>_{\Phi_W} & \mathbb P H^1(C, \mathrm{Sym}^2 W^* \otimes L)  \ar@{^{(}->}[u]\\
}
\end{eqnarray*}
\end{proposition}

\begin{remark}
The upper arrow $\Psi_W$ is the natural mapping of the ruled variety $\PP (\Endo (W))$, which is not necessarily defined everywhere at this stage, but it will turn out in \S 5 to be a morphism and furthermore an embedding under certain genus assumption. On the other hand, the lower arrow $\Phi_W$ is a morphism by Proposition  \ref{symplecticsmooth}.
\end{remark}

\begin{proof} 
%

Recall that $\ker \omega^{\theta}$  is  a subspace of $V_x$ of codimension two, and after we put 
$U^{\perp} = \ker \omega^{\theta}_x  \subset V_x  $,  the  family $\{V^{\ell} : [\ell] \in   \mathbb P(V_x/U^{\perp}) \}$  is a deformation of $W^* =V^{\ell_0}$. 
Applying Proposition \ref{KScomp} to  the kernel  of $\omega^{\theta}$, we get that
the Kodaira-Spencer map
$$T_{\ell_0}(\mathbb P(V_x/U^{\perp})) = \Hom(\ell_0^{\perp}/U^{\perp}, V_x/\ell_0^{\perp}) \rightarrow H^1(C, \End(V^{\ell_0}))$$
for the  family $\{V^{\ell} : [\ell] \in   \mathbb P(V_x/U^{\perp}) \}$,   is given by the composition
\begin{eqnarray*}
\Hom(\ell_0^{\perp}/U^{\perp}, V_x/\ell_0^{\perp}) \stackrel{\widehat{\beta_x}}{\rightarrow} \Hom(V_x^{\ell_0}, V_x/\ell_0^{\perp}) \stackrel{\delta}{\rightarrow}  H^1(C, \mathrm{End}(V^{\ell_0})),
\end{eqnarray*}
where  $\widehat{\beta_x} $ and $\delta$ are induced from
$$0 \rightarrow V^{\ell_0} \stackrel{\beta}{\rightarrow} V \rightarrow (V_x/\ell_0^{\perp}) \otimes \mathcal O_x \rightarrow 0.$$
Furthermore,  as in the case of the moduli space of vector bundles discussed in \cite{H00}, 
 the element $\delta (\widehat{\beta}_x(v))$ in $H^1(C, \mathrm{End}(V^{\ell_0}))$ for $v \in \Hom(\ell_0^{\perp}/U^{\perp} , V_x/\ell_0^{\perp})$, 
 is represented by the cocycle $\left\{\dfrac{e_2^* \otimes e_1 }{z} \text{ on } \mathcal U_0 \cap \mathcal U_j \right \}$. 
%
Here,
 $\{\mathcal U_0, \mathcal U_1, \dots, \mathcal U_N\}$ is a coordinate covering of $C$ such that
 \begin{itemize}
\item all the involved vector bundles are trivial on $\mathcal U_0$ and $\mathcal U_j$  for  $1 \leq j \leq N$, 
\item $x \in \mathcal U_0$ and  $x \not \in \mathcal U_j$ for  $1 \leq j \leq N$ 
so that on each $\mathcal U_j$,  we identify $ V^{\ell_0}$ and $V$ via $V^{\ell_0} \stackrel{\beta}{\rightarrow} V$,
\item  $z$ is a coordinate on $\mathcal U_0$ centered at $x$,
\item   $\{e_1, e_2, \dots, e_n\}$  and $\{f_1, f_2, \dots, f_n\}$ are  frames  of $V^{\ell_0}|_{\mathcal U_0}$ and  $V |_{\mathcal U_0}$, respectively  such that
 $e_{1,x} \in \theta \subset   V_x^{\ell_0} =W^*_x $, $f_{2,x} \in \ell_0^{\perp}/U^{\perp} \subset V_x /U^{\perp} =(W_x^{\theta})^*/U^{\perp}$ and
\item $\beta$ sends $ e_1, e_2, \dots, e_n$  to $ zf_1, f_2, \dots, f_n$. 
\end{itemize}

Thus $\delta (\widehat{\beta}_x(v))$ corresponds to the image of  $\iota(\theta)= \theta \otimes \omega(\theta, \,\cdot\,)$ via the duality $H^1(C, \Endo (W)) \simeq H^0(C, K_C  \otimes \Endo (W))^*$ induced by the residue pairing.
\end{proof}

\subsection{Orthogonal  bundles}
 By a similar argument as \cite[Lemma 2.2]{BH21}, the tangent space of $\MO$ at a regularly stable orthogonal bundle $W$ is given by $H^1(C, L \otimes \wedge^2W^*)$. In this context, we put $\ad (W) = L \otimes \wedge^2W^*$.
 
\begin{proposition} \label{prop: tangent map orthogonal}
Assume $g \ge 5$ and $n \ge 5$ as in Proposition \ref{orthosmooth}.
Let $\mathcal K$ be  the minimal rational component consisting of orthogonal  Hecke curves on $\MO$. 
%
Then for a generic $W \in \MO$, the tangent map $\tau_{W}$ 
$$\tau_{W} \colon \mathcal K_{W}= IG(2, W^*)  \to \mathbb P T_{W} \MO = \mathbb P H^1(C, \wedge^2 W^* \otimes L) $$
is  the composition $\Phi_W \circ \iota$,
where $\Phi_W$ is given by the complete linear system $|\mathcal O(1) \otimes \pi^* K_C|$. Also, $\iota(v_1 \wedge v_2) =b(v_1, \,\cdot\,) \otimes v_2 - b(v_2, \,\cdot\,) \otimes v_1$  for $[v_1 \wedge v_2] \in IG(2, W^*_x ) $
and  the image of $\iota(v_1\wedge v_2) $ in $\mathbb P(H^1(C, \wedge ^2W^* \otimes L))$ is given by the linear functional $H^0(C, \wedge ^2 W^*\otimes L^* \otimes K_C) \rightarrow K_{C,x}$ taking the trace of endomorphisms of $W_x$.
\begin{eqnarray*}
\xymatrix@R+1pc@C+2pc{
\mathcal K_{W}=IG(2, W^* ) \ar[dr]_{\iota}  \ar[r]  & \mathbb P(\Endo (W))  \ar@{-->}[r]_{\Psi_W}   &   \mathbb P H^1(C, \Endo W)  \\
 & \mathbb P(\wedge^2 W^* \otimes L ) \ar@{^{(}->}[u] \ar[r]<4pt>_{\Phi_W} & \mathbb P H^1 (C, \wedge ^2 W^* \otimes L) \ar@{^{(}->}[u] \\
}
\end{eqnarray*}
\end{proposition}

\begin{remark}
As before, the upper arrow $\Psi_W$ is  not necessarily defined everywhere at this stage, but it will turn out in \S 5 to be a morphism and furthermore an embedding under certain genus assumption. On the other hand, the lower arrow $\Phi_W$ is a morphism by Proposition  \ref{orthosmooth}.
\end{remark}

\begin{proof} 
%

Recall that the kernel  of $b^{\Theta}$ is  a subspace of codimension 4, and if we put
$U^{\perp} = \Ker b^{\Theta}_x \subset V_x = (W_x^{\Theta})^*$, then the   family $\{V^{\wp} : [\wp] \in   IG(2,  V_x/U^{\perp}) \}$  is a deformation of $W^* =V^{\wp_0}$. 
As in the proof of Proposition \ref{orthosmooth}, we apply  Proposition \ref{prop: KS map codimension 4}
to  the kernel  of $b^{\Theta}$ and  we get that
  the Kodaira-Spencer map
$$T_{\wp_0}(Gr(2,V_x/U^{\perp})) = \Hom(\wp_0^{\perp}/U^{\perp}, V_x/\wp_0^{\perp}) \rightarrow H^1(C, \End(V^{\wp_0}))$$
for the   family $\{V^{\wp} : [\wp] \in   Gr(2,  V_x/U^{\perp}) \}$ is given by the composition
\begin{eqnarray*}
\Hom(\wp_0^{\perp}/U^{\perp}, V_x/\wp_0^{\perp}) \stackrel{\widehat{\beta_x}}{\rightarrow} \Hom(V_x^{\wp_0}, V_x/\wp_0^{\perp}) \stackrel{\delta}{\rightarrow}  H^1(C, \mathrm{End}(V^{\wp_0})),
\end{eqnarray*}
where  $\widehat{\beta_x} $ and $\delta$ are induced from
$$0 \rightarrow V^{\wp_0} \stackrel{\beta}{\rightarrow} V \rightarrow (V_x/\wp_0^{\perp}) \otimes \mathcal O_x \rightarrow 0.$$

Choose a coordinate covering $\{\mathcal U_0, \mathcal U_1, \dots, \mathcal U_N\}$ of $C$ such that
\begin{itemize}
\item all the involved vector bundles are trivial on $\mathcal U_0$ and  $\mathcal U_j$ for  $1 \leq j \leq N$,
\item  $x \in \mathcal U_0$ and $x \not \in \mathcal U_j$ for  $1 \leq j \leq N$ so that on each  $\mathcal U_j$, we  identify $ V^{\wp_0}=W^*$ and $V=(W^{\Theta})^*$ by $V^{\wp_0} \stackrel{\beta}{\rightarrow} V$,
\item $z$ is a coordinate on $\mathcal U_0$ centered at $x$, 
\item $\{e_1, e_2, \dots, e_n\}$ and $\{f_1, f_2, \dots, f_n\}$  are frames of $V^{\wp_0}|_{\mathcal U_0}$ and $V |_{\mathcal U_0}$ respectively such that   $e_{1,x}, e_{2,x} \in \Theta \subset  V_x^{\wp_0}=W^*_x $, $f_{3,x}, f_{4,x} \in \wp_0^{\perp}/U^{\perp} \subset V_x /U^{\perp} =(W_x^{\Theta})^*/U^{\perp}$ and
\item $\beta$ sends $ e_1, e_2, e_3, e_4 $ to $zf_1, zf_2, f_3, f_4 $.
\end{itemize}
Then
the tangent space $T_{\wp_0}(IG(2,V_x/U^{\perp}))$ is generated by   
\[
v= f_{3,x}^* \otimes f_{1,x} - f_{4,x}^* \otimes f_{2,x} \in \Hom(\wp_0^{\perp}/U^{\perp} , V_x/\wp_0^{\perp}).
\]
 Furthermore,  $v\circ \beta$ maps
$e_{1,x}$, $e_{2,x}$, $e_{3,x}$, $e_{4,x}$ to $0$, $0$, $f_{1,x}$, $f_{2,x}$ up to a constant multiple. Thus $\widehat{\beta_x}(v)$ can be extended to $\tilde{v}$ with $\tilde v(e_1) =0$ and $\tilde{v}(e_2)=0$ and $\tilde v(e_3) =f_1$ and $\tilde{v}(e_4)=f_2$.  Then $\delta (\widehat{\beta_x}(v))$ is defined by the cocycle
$$\widehat{v}_j=(\beta|_{\mathcal U_0 \cap \mathcal U_j})^{-1} \circ \widetilde v \ \in H^0(\mathcal U_0  \cap \mathcal U_j, \End V^{\ell_0}) ,$$
where
 $\beta|_{\mathcal U_0 \cap \mathcal U_j}   \colon V^{\ell_0}|_{\mathcal U_0 \cap \mathcal U_j} \rightarrow V|_{\mathcal U_0 \cap \mathcal U_j}
 $
 is the isomorphism. Therefore,  $\delta (\widehat{\beta}_x(v))$ 
 is represented by the cocycle $\left\{\dfrac{e_3^* \otimes e_1 -e_4^* \otimes e_2}{z} \text{ on } \mathcal U_0 \cap \mathcal U_j \right \}$, where $\{e_1^*,e_2^*, \dots, e_n^*\}$ is the dual frame. 
 
Thus $\delta (\widehat{\beta}_x(v))$ corresponds to the image of  $\iota(v_1 \wedge v_2) =b(v_1, \,\cdot\,) \otimes v_2 - b(v_2, \,\cdot\,) \otimes v_1$  via the duality $H^1(C, \Endo (W)) \simeq H^0(C, K_C  \otimes \Endo (W))^*$ induced by the residue pairing.

\end{proof}
 
 \subsection{Simplicity of the tangent bundles}
 Now we can prove the simplicity of the tangent bundles. Let us denote by $\M^\circ$  the smooth locus of the moduli space $\M S_C(2r,L)$ ($\MO$, respectively)  of $L$-valued symplectic (orthogonal, respectively) bundles on $C$ of genus $g$.

 We exclude the case of symplectic bundles of rank two, which are nothing but vector bundles of rank two. In this case, the simplicity of the tangent bundle of $\mathcal S U_C(2,d)$ has been proven in \cite{H01}.  Moreover when $d$ is odd,   the stability of the tangent bundle   has been proven in \cite{H00}. Also recall that the rank of orthogonal bundles are assumed  to be at least 5  to guarantee the minimality of the orthogonal Hecke curves (see \S 2.5). 
 
\begin{theorem}
Assume that $g \ge 4 , n \ge 4$ for symplectic case, and $g \ge 5, n \ge 5$ for orthogonal case.  
Then the tangent bundle $T(\M^\circ)$  is simple. 
\end{theorem}

\begin{proof}
By Proposition  \ref{prop: tangent map symplectic} and Proposition \ref{prop: tangent map orthogonal}, the tangent maps of symplectic and orthogonal Hecke curves are given by the corresponding complete linear systems. Hence for each case, the VMRT at a generic point is  non-degenerate, and the wanted result follows from  Proposition \ref{nondegsimple}. 
\end{proof}

\section{Biregularity of the tangent map}

Let $W$ be an $L$-valued symplectic or orthogonal bundle over $C$. Throughout this section, we will assume that an orthogonal bundle of even rank, say $2r$, admits an isotropic subbundle of rank $r$. As shown in \cite[Lemma 2.5]{CCH21}, this is equivalent to that $\det (W) = L^r$. Accordingly, $\mathcal M O_C(2r, L)$ denotes one of the moduli components which parameterizes  $L$-valued orthogonal bundles of rank $2r$ with determinant $L^r$. On the other hand, every orthogonal bundle of odd rank, say $2r+1$, admits an isotropic subbundle of rank $r$ by \cite[Lemma 2.7]{CCH21}.  Hence $\mathcal M O_C(2r+1, L)$ denotes any  moduli component which parameterizes   $L$-valued orthogonal bundles of rank $2r+1$.

As discussed in the previous section, we write
\[ \ad W \ = \ \begin{cases} \Sym^2 W  \otimes L^* \hbox{ if $W$ is symplectic;} \\ \wedge^2 W \otimes L^* \hbox{ if $W$ is orthogonal.} \end{cases} \]
In either case, $\ad W$ is a self-dual subbundle of  $\Endo W$.
 
Let $\pi \colon \PP (\ad W) \to C$ be the associated projective bundle of $\ad W$.  We consider the natural map
\[ \Phi_W \colon \PP( \ad W ) \ {\dashrightarrow} \ \PP H^0 ( C, \Kc \otimes \ad W )^* , \]
given by the complete linear system 
$| \cO_{\PP ( \ad W)} (1) \otimes \pi^* \Kc | $.
Our goal will be to prove that $\Phi_W$ is an embedding.   In fact, we show a stronger statement: the map
\[ \Psi_W \colon \PP( \Endo W ) \ \dashrightarrow \ \PP H^0 ( C, \Kc \otimes \Endo W )^* \]
 is an embedding. More precisely:

\begin{theorem}  \label{embedding} 
Let $\ell := \deg (L) \in \{0,1 \}$. The map $\Psi_W$ is an embedding for:\\
(1) a generic $W \in \mathcal M S_C(2r, L)$ if $r \ge 2$ and $ g \ge 5 +2\ell$,  \\
(2) a generic $W \in \mathcal M O_C(2r, L)$  if either ($r =3, \: g \ge 8$) or ($r \ge 4, \: g \ge 7$), \\
(3) a generic $W \in \mathcal M O_C(2r+1, L)$ if either  ($r=2, \ g \ge 14$) or ($r \ge 3, \ g \ge 9$).  
\end{theorem}

\begin{corollary}
Let $\M$ be one of the moduli spaces in the above theorem. Let 
\[
\tau_W \colon \mathcal K_x \to \PP (T_{W} \M) 
\]
 be the tangent map of the minimal rational component $\mathcal K$ associated to  the symplectic or orthogonal Hecke curves. Then $\tau_W$ is an embedding of the corresponding VMRT under the same assumption as in Theorem \ref{embedding}.
\end{corollary}
\begin{proof}
This follows from the fact that $\Psi_W$, hence $\Phi_W$, is an embedding, together with the pictures of Proposition \ref{prop: tangent map symplectic} and Proposition \ref{prop: tangent map orthogonal}.
\end{proof}

The remaining parts of this section are devoted to proving Theorem \ref{embedding}.
Following \cite[Proof of Theorem 3.1]{HR04}, we shall use the fact that $\Psi_W$ is an embedding if and only if
\begin{equation} h^0 ( \Oc (D) \otimes \Endo W  ) = 0 \hbox{ for all } D \in \Ct , \label{embppty} \end{equation}
 where $C^{(2)}$ parameterizes the effective divisors of degree two. 
 
 We first show the following:
\begin{lemma} \label{open} Let $\cE \to B \times C$ be a family of vector bundles over $C$. Then the subset
\[ \left\{ b \in B : h^0 ( \Oc (D) \otimes \cE_b ) = 0 \hbox{ for all } D \in \Ct \right\} \]
is open in $B$, where $\cE_b = \cE|_{\{b \} \times C}$.
\end{lemma}

\begin{proof}
The complement of the locus in question is
\[ \left\{ b \in B : h^0 ( \Oc ( D ) \otimes \cE_b ) \ge 1 \hbox{ for some } D \in \Ct \right\} . \]
This is the image in $B$ of the closed set
\[ \left\{ ( b , D ) \in B \times \Ct : h^0 ( \Oc (D) \otimes \cE_b ) \ge 1 \right\} \]
by the projection from $B \times \Ct$. As $\Ct$ is projective and in particular complete, this projection is closed. 
 The statement follows.
\end{proof}

Applying Lemma \ref{open} to a suitable \'etale cover of each of the moduli spaces in question, we see that it suffices to exhibit a single bundle $V$ with property (\ref{embppty}). 
We now assemble some vanishing results which we shall use to this end.

\begin{lemma} \label{sufficient} Let $E_1$ and $E_2$ be vector bundles such that for all $D \in \Ct$ we have
\begin{enumerate}
\renewcommand{\labelenumi}{(\arabic{enumi})}
\item[(i)] $h^0 ( \Oc (D) \otimes \Endo E_1 ) = 0$
\item[(ii)] $h^0 ( \Oc (D) \otimes \Endo E_2 ) = 0$
\item[(iii)] $h^0 ( \Hom ( E_1, E_2 (D) )) = 0$
\item[(iv)] $h^0 ( \Hom ( E_2, E_1 (D) )) = 0$
\end{enumerate}
Let $\Pi \subseteq H^1 ( \Hom (E_2, E_1) )$ be a subspace of dimension at least $2 \cdot \rank ( E_1 ) \cdot \rank ( E_2 ) + 3$. Then if $0 \to E_1 \to W \to E_2 \to 0$ is an extension whose class $\delta$ is a general element of $\Pi$, we have $h^0 ( \Oc (D) \otimes \Endo W ) = 0$ for all $D \in \Ct$.
\end{lemma}

\begin{proof} Let $W$ be an extension whose class is a general element of $\Pi$. Let $D$ be an element of $\Ct$, and suppose that $\alpha \colon W \to W (D)$ is a nonzero map. We shall show that under the hypotheses above, $\alpha = \Iden_W \otimes \mathbf{s}$ for some $\mathbf{s} \in H^0 ( \Oc (D) )$. This will suffice to prove the statement.

By (iii), the restriction $\alpha |_{ E_1}$ factorizes via $E_1(D)$. Therefore, we have a diagram
\[ \xymatrix{ 0 \ar[r] & E_1 \ar[r] \ar[d] & W \ar[r]^q \ar[d]^\alpha & E_2 \ar[r] \ar[d] & 0 \\
 0 \ar[r] & E_1 (D) \ar[r] & W (D) \ar[r]^{q'} & E_2 (D) \ar[r] & 0. } \]
 By (i), then, $\alpha|_{E_1} = \Iden_{E_1} \otimes \mathbf{s}$ for some $\mathbf{s} \in H^0 ( \Oc (D) )$. (If $C$ is nonhyperelliptic then $(\mathbf{s}) = D$.) Therefore, $\alpha - \Iden_W \otimes \mathbf{s}$ is a map $W \to W (D)$ vanishing on $E_1$. Hence
\begin{equation} \alpha - \Iden_W \otimes \mathbf{s} \ = \ \beta \circ q \label{alphamso} \end{equation}
for some $\beta \in H^0 ( \Hom ( E_2 , W (D) ))$. 

If $\beta \neq 0$ then by (iv) we see that $q' \circ \beta$ is a nonzero map $E_2 \to E_2 (D)$. By (ii), we have
\[ q' \circ \beta \ = \ \Iden_{E_2} \otimes \mathbf{s}_{D'} \]
for some $D' \in |D|$ and $\mathbf{s}_{D'} \in H^0 ( \Oc (D) )$ satisfying $( \mathbf{s}_{D'} ) = D'$. (If $C$ is nonhyperelliptic then $D' = D$.) In particular, the map $\Iden_{E_2} \otimes \mathbf{s}_{D'} \colon E_2 \to E_2 (D')$ lifts to $W$. This means that
\[ \delta ( W ) \ \in \ \ker \left( \Iden_{E_2} \otimes \mathbf{s}_{D'} \colon H^1 ( \Hom ( E_2 ( D ) , E_1 ( D ) ) ) \to H^1 ( \Hom ( E_2 ( D - D' ), E_1 ( D ) ) ) \right) . \]
In view of the exact sequence
\[ 0 \ \to \ E_2^* \otimes E_1 \ \to \ E_2^* \otimes E_1 (D') \ \to \ E_2^* \otimes E_1 (D')|_{D'} \ \to \ 0 , \]
this kernel has dimension at most $2 \cdot \rank (E_2) \cdot \rank (E_1)$. The union
\[ \bigcup_{D' \in \Ct} \ker \left( \left( \Iden_{E_2} \otimes \mathbf{s}_{D'} \right)^* \right) \]
is therefore of dimension at most $2 \cdot \rank (E_2) \cdot \rank (E_1) + 2 < \dim \Pi$. But since $\delta(W)$ was assumed to be general in $\Pi$, we may assume that the map $\Iden_{E_2} \otimes \mathbf{s}_{D'}$ does not lift to $W$. Therefore, we must have $\beta = 0$.

By (\ref{alphamso}), we obtain $\alpha = \Iden_W \otimes \mathbf{s}$, as desired. \end{proof}

\begin{lemma} \label{rovan} Suppose $t$ and $s$ are integers with $s \ge 0$ and $t + 2s < g$. Let $N$ be a generic line bundle  in $\Pic^t (C)$. Then $h^0 ( N ( D )) = 0$ for all $D \in C^{(s)}$, where $C^{(s)}$ parameterizes the effective divisors of degree $s$. \end{lemma}

\begin{proof} If $t + s < 0$ then this is clear. Otherwise: 
If $h^0(C, N(D)) \neq 0$ for some $D \in C^{(s)}$, then $N$ is of the form $\Oc (D_1-D)$ for some $D_1 \in C^{(t+s)}$ and $D \in C^{(s)}$. Hence the locus of such $N$ is of dimension at most $t+2s$ in $\Pic^t(C)$. (We take $C^{(0)} = \{ \Oc \}$.) By hypothesis, this locus is not dense in $\Pic^t(C)$. The statement follows. 
\end{proof}

\begin{lemma} \label{Segre}
Let $r, s, t$ be integers with $r \ge 2$,  $s \ge 0$ and $t +rs < (r-1)(g-1)$.
Let $G$ be a generic stable bundle of rank $r$ and degree $t$. Then $h^0(G(D)) =0$ for all $D \in C^{(s)}$.
\end{lemma}

\begin{proof}
A nonzero section of $G(D)$ gives   a sheaf injection $\Oc (-D) \to G$. This implies that the first Segre invariant of $G$ is bounded by
\[ s_1 ( F  ) \ \le \  \deg (G) - \rank (G) \cdot \deg (\Oc(-D)) \ = \ t + rs  . \] 
But as $G$ is generic, by \cite[Satz 2.2]{L83} we have
\[ s_1 ( G  ) \ \ge \ ( r-1 ) ( g-1 ) . \]  
Hence we have an  inequality $(r-1)(g-1) \le t+rs$, which contradicts the assumption.  
Thus     $h^0 ( C, G(D) ) = 0$ for all $D$.
  \end{proof}

\begin{lemma} \label{eot}  Let $E$ be  a generic stable bundle  over $C$ of rank $r \ge 1$ and degree $e$. \\
(1) If  $e\ge 0, s \ge 1$ and $2e +2s < g$, then  $h^0 ( C, E \otimes E (D) ) = 0$ for all $D \in C^{(s)}$.\\
(2) If $ e<0, s \ge 1$ and $2s <g$,  then  $h^0 ( C, E \otimes E (D) ) = 0$ for all $D \in C^{(s)}$.
  \end{lemma}

\begin{proof} 
By Lemma  \ref{open}, the locus of points $E $ which satisfy the desired vanishing property is open.
Thus  it suffices to exhibit a single $E $ with the desired property. (Note that even if such an $E$ is unstable, we can continuously deform it to get a generic stable one with the same vanishing property.)

We  first assume $e \ge 0$ and proceed by induction on $r  $. For $r = 1$, this holds by Lemma \ref{rovan}. 
Suppose now that $r \ge 2$. By induction hypothesis, there is  a stable vector bundle $F$ of rank $r-1$ and degree $e$ with the desired vanishing property, which can be assumed to be  generic in the moduli.
Fix a generic line bundle $L_0 \in \Pic^0(C)$ and put  $E = L_0 \oplus F$.  Then  $E \otimes E (D)$ has four direct summands, and it suffices to show all of them have no sections. 



By induction we may assume that $ h^0 ( C, F \otimes F ( D ) ) = 0$ for all $D \in \Ct$.
Also we note that $h^0 ( C, (L_0)^2 (D) ) = 0$ for all $D \in \Ct$  by Lemma \ref{rovan}.
Finally we show that $h^0 ( C, F \otimes L_0 (D) ) = 0$ for all $D$. If $F$ is a line bundle, the  condition for the vanishing is $e+2s <g$ by Lemma \ref{rovan}, and this holds by assumption. If $\rank (F) = r-1 \ge 2$, then   the  condition for the vanishing is $e+(r-1)s < (r-2)(g-1)$ by Lemma \ref{Segre}. Also this holds by assumption, noting that $s \ge 1$ and $r  = \rank (F) +1 \ge 3$.

It follows that for a generic extension $E$ given above, we have $h^0 ( C, E \otimes E (D) ) = 0$ for all $D \in C^{(s)}$.

To finish, we need to find a bundle of degree $e<0$ with the vanishing property. 
By the part (1), we may choose a stable bundle $E $ of rank $r$ and degree 0  satisfying $h^0(E \otimes E (D)) =0$ for all $D \in C^{(s)}$. Let $\widetilde{E} $ be obtained by an elementary transformation $0 \to \widetilde{E} \to E \to \tau \to 0$, where $\deg (\tau) = e$. Then $\widetilde{E} \otimes \widetilde{E} (D)$ is a subsheaf of $E \otimes E ( D )$, and by the vanishing result for $E$, we obtain $h^0 ( C, \widetilde{E} \otimes \widetilde{E} (D) ) = 0$. This completes the proof.
 \end{proof}
 
  Also we record the following cohomology vanishing result \cite[Proposition 3.2]{HR04}, which will be used later:
 \begin{proposition} \label{HRvanishing}
 \ For a general stable vector bundle $F$ of arbitrary rank and degree, $H^0(C, (\End_0F) (D)) = 0$ for any effective divisor $D$ of degree $d$ whenever $g \ge \frac{3}{2}d +2$. \qed
 \end{proposition}
 
 Now we apply these results to  show the desired vanishing property for a generic symplectic and orthogonal bundle of even rank. Consider an extension  
 \[
(\ast) \ \ \ \ \ 0 \to E \to W \to E^* \otimes L \to 0.
\]
for $E \in \rSU ( r, e )$ which is a generic stable bundle.   Recall that, by \cite[Criterion 2.1]{H07}, we get a  symplectic  bundle $W$ if we choose $(\ast)$  in $H^1 ( \Sym^2 E  \otimes L^* )$, and an orthogonal  bundle if we choose $(\ast)$ in $H^1 ( \wedge^2 E  \otimes L^* )$. Also in both cases, $E \to W$ is an isotropic subbundle. (From now on whenever we discuss  orthogonal bundles of even rank, we consider those bundles with an isotropic subbundle of the half rank only.)

  Once we find a symplectic/orthogonal bundle $W$ in this extension which has the desired vanishing property  (\ref{embppty}), by deformation this will show that the vanishing property holds for a general stable symplectic/orthogonal bundle with the same topological invariants.  The only topological invariants of a symplectic bundle are rank and degree, while  we need to additionally consider the 2nd Stiefel--Whitney class  for an  orthogonal bundles. Note that we may assume $L= \cO_C$ when $\deg(L)$ is even and $L = \cO_C(x)$ for some $x \in C$ when $\deg(L)$ is odd.

If $L = \cO_C$ so that $\deg (W) =0$, the moduli space $\mathcal M O_C(2r,  \cO_C)$ has two  components classified by the 2nd Stiefel--Whitney class $w_2(W)$ such that the degree of any rank $r$ isotropic subbundle of $W$ has the same parity as $w_2(W)$. On the other hand, if $L = \cO_C(x)$ so that $\deg (W)=r$, then  by \cite[\S2]{CCH21} every orthogonal bundle $W \in \mathcal M O_C(2r,  \cO_C(x))$ has rank $r$ isotropic subbundles both of even degree and of odd degree.

 Hence it suffices to show that  a bundle $W$ obtained by a generic extension $(\ast)$  has the vanishing property  (\ref{embppty}) in each of the following cases:
\begin{itemize}
\item For symplectic bundles: $e = \deg E = 0$ and  $L= \cO_C$ or $\cO_C(x)$.
\item For orthogonal bundles:  ($L= \cO_C$ and $e =   -1,0$) or ($L = \cO_C(x)$ and $e=0$). 
\end{itemize}

\begin{proposition} \label{exhibitV}    Let $\ell = \deg (L) \in \{0,1\}$.
\begin{enumerate}
\item    For $r \ge 2$, suppose $g \ge 5+2 \ell$.   If  $(\ast)$ is a symplectic extension defined by a generic class in $H^1 ( \Sym^2 E  \otimes L^* )$ where $e:=\deg E = 0$, then $h^0 ( \Oc (D) \otimes \Endo W ) = 0$ for all $D \in \Ct$. \\ 
\item   Suppose  either ($r =3, \: g \ge 8$) or ($r \ge 4, \: g \ge 7$).  If $(\ast)$ is  an orthogonal extension defined by a generic class in $H^1 ( \wedge^2 E  \otimes L^*)$    where  $e:=\deg E \in \{-1, 0\}$,  then $h^0 ( \Oc (D) \otimes \Endo W ) = 0$ for all $D \in \Ct$.
\end{enumerate} 
\end{proposition}

\begin{proof}  

For the vanishing, let us apply Lemma \ref{sufficient} with $E_1 = E^*$ and $E_2 = E \otimes L$. As $E$ is general in moduli and $g \ge 5$, conditions (i) and (ii) follow from Proposition \ref{HRvanishing}.  

Conditions (iii)  and (iv) read $h^0(E \otimes E \otimes L(D)) =0$ and $h^0(E^* \otimes E^* \otimes L^*(D)) =0$  for any $D \in C^{(2)}$, respectively. These vanishing conditions are checked by Lemma \ref{eot}  under the assumption $g \ge 5 +2 \ell$  in each case for $L = \cO_C$ or $L= \cO_C(x)$.   
 To get a conclusion from Lemma \ref{sufficient}, it will suffice to check that $h^1(\Sym^2 E \otimes L^*)$ and  $h^1 ( \wedge^2 E \otimes L^*)$ are  bigger than $ 2 r^2 + 2$, respectively. 
 
 In the symplectic case, by Riemann--Roch the desired inequality would follow  from 
\[ 2r^2 + 3 \ \le \  \frac{1}{2}r(r+1)(g-1+\deg(L)). \]
 This reads
 \[
   \frac{ g-1+\ell}{2} \ge  \frac{2r^2+3 }{r(r+1)}, 
\]
 which holds for $r \ge 2$ and $g \ge 5 - \ell$.  

In the orthogonal case, the desired inequality would follow from
 \[ 2r^2 + 3 \ \le \  -(r-1) e + \frac{1}{2}r(r-1)(g-1+\deg(L)). \]
 This reads
 \[
   \frac{ g-1 + \ell}{2} \ge \frac{2r^2+3 +(r-1)e}{r(r-1)}   ,
\]
    which holds if   either ($r =3, \: g \ge 8$) or ($r \ge 4, \: g \ge 7$). 
\end{proof}

Now we discuss the case of orthogonal bundles of odd rank. In this case, we require some more vanishing results. To get a better genus bound, we assume here $e= \deg E \in \{-2, -1 \}$ instead of $\{-1, 0 \}$.

\begin{lemma} \label{npovan} Suppose $g \ge 8$. Suppose $r \ge 2$, and let $E$ be a generic stable bundle of rank $r$ and degree $e$, where $e \in \{  -2,-1 \}$. Let $0 \to E \to F \to \Oc \to 0$ be a generic extension. Then $h^0 ( \Oc(D) \otimes \Endo F ) = 0$ for all $D \in \Ct$. \end{lemma}

\begin{proof} Let us apply Lemma \ref{sufficient} with $E_1 = E$ and $E_2 = \Oc$. Condition (i) follows as above from \cite[Proposition 3.2]{HR04} since $E$ is general and $g \ge 5$. Condition (ii) is trivial. Conditions (iii) and (iv) follow from Lemma \ref{Segre} under the assumption $g \ge 8$.

It remains to check that $\dim H^1 ( \Hom ( \Oc , E ) ) > 2 r + 3$. As $E$ is stable of negative degree, $h^0 ( E ) = 0$. By Riemann--Roch, $h^1 ( E ) = -e + r (g-1) \ge r (g-1) + 1$. Using the inequalities $r \ge 2$ and $g \ge 5$, one checks that this exceeds $2r + 3$. 
 The statement now follows from Lemma \ref{sufficient}. \end{proof}

\begin{lemma} \label{npovan2} Let $E$ and $F$ be as in Lemma \ref{npovan}, and suppose $g \ge 8$. Then for all $D \in \Ct$ we have $h^0 ( E^* \otimes F^* (D) ) = 0$ and $h^0 ( F \otimes E (D) ) = 0$. \end{lemma}

\begin{proof} By construction of $F$, for each $D \in \Ct$ we have exact sequences
\[ H^0 ( E \otimes E (D) ) \ \to \ H^0 ( F \otimes E (D) ) \ \to \ H^0 ( E (D) ) \ \to \ \cdots \]
and
\[ H^0 ( E^* (D) ) \ \to \ H^0 ( E^* \otimes F^* (D) ) \ \to \ H^0 ( E^* \otimes E^* (D) ) \ \to \ \cdots \]
By Lemma \ref{eot}, we have $h^0 ( E \otimes E (D) ) = 0 = h^0 ( E^* \otimes E^* (D) )$ for all $D \in \Ct$. The vanishing $h^0 ( E^* (D) ) = 0 = h^0 ( E(D) )$ is a consequence of Lemma \ref{Segre}. The statement follows. \end{proof}

By \cite[Lemma 2.4]{CCH21}, for any $L$-valued orthogonal bundle $W$ of odd rank, there is a line bundle $N$  such that $W \otimes N$ is an $\cO_C$-valued orthogonal bundle of trivial determinant. So we may work only for the moduli component  of $\cO_C$-valued orthogonal bundles of trivial determinant. As in the even rank case, the moduli space $\mathcal M O_C(2r+1,  \cO_C)$ has two  components classified by the 2nd Stiefel--Whitney class $w_2(W)$.  By \cite[Theorem 3.1]{CH15}, the degree of any rank $n$ isotropic subbundle of $W$ has the same parity as $w_2(W)$.
To construct such orthogonal bundles of rank $2r+1$ as extensions, we use some results from \cite[{\S} 3]{CH15}. Let $0 \to E \xrightarrow{j} F \to \Oc \to 0$ be an extension as above, and let $\Pi_j$ be the subspace of $H^1 ( F \otimes E )$ as defined in \cite[{\S} 3]{CH15}, which contains $H^1 ( C, \wedge^2 E )$ as a subspace of codimension 1.
 By \cite[Lemma 3.2]{CH15}, an extension $0 \to E \to W \to F^* \to 0$ defined by a class  contained in  $ \Pi_j \setminus H^1 ( C, \wedge^2 E ) $ is an orthogonal bundle.

\begin{proposition} \label{exhibitOdd} Suppose either  ($r=2, \ g \ge 14$) or ($r \ge 3, \ g \ge 9$). Let $0 \to E \to W \to F^* \to 0$ be a stable orthogonal bundle of rank $2r+1$ as above, whose extension whose class is general in $\Pi_j$. Then $h^0 ( \Oc (D) \otimes \Endo W ) = 0$ for all $D \in \Ct$. \end{proposition}

\begin{proof} Again, we use Lemma \ref{sufficient}; this time with $E_1 = E$ of rank $r$, degree $e$  and $E_2 = F^*$ of rank $r+1$, degree $-e$  where $e \in \{ -2, -1\}$. Condition (i) follows from Proposition \ref{HRvanishing}  as before. Condition (ii) follows from Lemma \ref{npovan}. Conditions (iii) and (iv) follow from Lemma \ref{npovan2}.

Lastly, we must show that $\dim ( \Pi_j ) \ge 2 \cdot \rank ( E ) \cdot \rank (F^*) + 3 = 2r(r+1) + 3$. By Riemann--Roch,
\[ \dim ( \Pi_j ) \ = \ h^1 ( \wedge^2 E )  +1\ \ge \ -(r-1)e + \frac{r(r-1)}{2} (g-1)+1  . \]
Hence it suffices to have
\[
 -(r-1)e + \frac{r(r-1)}{2} (g-1)+1 \ \ge \ 2r(r+1) + 3.
\]
 Simplifying this by using $e \le -1$,  we get the wanted inequality if 
\begin{equation} \label{ineqPij}
\frac{r(r-1)}{2} (g-1) \ge 2r^2 + r+3.
\end{equation}
This holds for ($r=2, g \ge 14$) or ($r \ge 3, g \ge 9$). 
\end{proof}

\begin{proof}[Proof of Theorem \ref{embedding}] By Lemma \ref{open}, it suffices to exhibit a single element $W_0$  of each of the moduli spaces satisfying the vanishing property (\ref{embppty}). This follows from Proposition \ref{exhibitV} and \ref{exhibitOdd}.
 \end{proof}
 
 \begin{remark}
(1)  The genus bound in Theorem \ref{embedding} is not optimal, and can be improved simply by computing inequalities above more accurately. For instance, the above inequality (\ref{ineqPij}) reads $g \ge  \frac{2(2r^2+r+3)}{r(r-1)} +1$, which  becomes $g \ge 6$ in the limit $r \to \infty$.\\
(2) We did not consider orthogonal bundles of rank $\le 4$ according to the convention in \S 2.5. But all the arguments in \S 5 are valid for arbitrary rank, hence we can apply the same argument to get a very ampleness result for orthogonal bundles of rank $\le 4$.  For example, we could state Proposition \ref{exhibitV}  (2) for $r=2$, in which case the genus assumption would  be  $g \ge 12$.   Hence the map $\Psi_W$ is an embedding for a generic $W \in \mathcal M O_C(4, L)$ if $g \ge 12$. \\
(3) Theorem \ref{embedding} shows that the VMRT of $\MS$ and $\MO$ at a generic point $W$ is biregular to $\PP (W^*)$ and $IG(2, W)$ respectively, under the assumption on the genus. As remarked in \cite[Remark 6.2]{CCL22}, this improves the involved genus bound in \cite[\S 6]{CCL22}.
\end{remark}

 \bibliographystyle{amsplain}
\bibliography{simplicity-vmrt}

\providecommand{\bysame}{\leavevmode\hbox to3em{\hrulefill}\thinspace}
\providecommand{\MR}{\relax\ifhmode\unskip\space\fi MR }
\providecommand{\MRhref}[2]{%
  \href{http://www.ams.org/mathscinet-getitem?mr=#1}{#2}
}
\providecommand{\href}[2]{#2}
\begin{thebibliography}{10}

\bibitem{BH21}
Ali Bajravani and George~H. Hitching, \emph{Brill-{N}oether loci on moduli
  spaces of symplectic bundles over curves}, Collect. Math. \textbf{72} (2021),
  no.~2, 443--469. \MR{4248596}

\bibitem{CCH21}
Daewoong Cheong, Insong Choe, and George~H. Hitching, \emph{Isotropic {Q}uot
  schemes of orthogonal bundles over a curve}, Internat. J. Math. \textbf{32}
  (2021), no.~8, Paper No. 2150047, 36. \MR{4300437}

\bibitem{CCL22}
Insong Choe, Kiryong Chung, and Sanghyeon Lee, \emph{Minimal rational curves on
  the moduli spaces of symplectic and orthogonal bundles}, J. Lond. Math. Soc.
  (2) \textbf{105} (2022), no.~1, 543--564.

\bibitem{CH15}
Insong Choe and George~H. Hitching, \emph{Maximal isotropic subbundles of
  orthogonal bundles of odd rank over a curve}, Internat. J. Math. \textbf{26}
  (2015), no.~13, 1550106, 23. \MR{3435964}

\bibitem{CH22}
\bysame, \emph{Low rank orthogonal bundles and quadric fibrations},
  arXiv:2203.06645 (2022).

\bibitem{H07}
George~H. Hitching, \emph{Subbundles of symplectic and orthogonal vector
  bundles over curves}, Math. Nachr. \textbf{280} (2007), no.~13-14,
  1510--1517. \MR{2354976}

\bibitem{H00}
Jun-Muk Hwang, \emph{Tangent vectors to {H}ecke curves on the moduli space of
  rank 2 bundles over an algebraic curve}, Duke Math. J. \textbf{101} (2000),
  no.~1, 179--187.

\bibitem{H01}
\bysame, \emph{Hecke curves on the moduli space of vector bundles over an
  algebraic curve}, Algebraic geometry in {E}ast {A}sia ({K}yoto, 2001), World
  Sci. Publ., River Edge, NJ, 2002, pp.~155--164.

\bibitem{HR04}
Jun-Muk Hwang and S.~Ramanan, \emph{Hecke curves and {H}itchin discriminant},
  Ann. Sci. \'{E}cole Norm. Sup. (4) \textbf{37} (2004), no.~5, 801--817.

\bibitem{L83}
Herbert Lange, \emph{Zur {K}lassifikation von {R}egelmannigfaltigkeiten}, Math.
  Ann. \textbf{262} (1983), no.~4, 447--459. \MR{696517}

\bibitem{NR75}
M.~S. Narasimhan and S.~Ramanan, \emph{Deformations of the moduli space of
  vector bundles over an algebraic curve}, Ann. of Math. (2) \textbf{101}
  (1975), 391--417.

\bibitem{NR78}
\bysame, \emph{Geometry of {H}ecke cycles. {I}}, C. {P}. {R}amanujam---a
  tribute, Tata Inst. Fund. Res. Studies in Math., vol.~8, Springer, Berlin-New
  York, 1978, pp.~291--345.

\bibitem{S05}
Xiaotao Sun, \emph{Minimal rational curves on moduli spaces of stable bundles},
  Math. Ann. \textbf{331} (2005), no.~4, 925--937.

\end{thebibliography}

\end{document}